\documentclass[11pt,reqno]{amsart}
\usepackage[latin9]{inputenc}
\usepackage[T1]{fontenc}
\usepackage{palatino}
\usepackage[a4paper]{geometry}
\usepackage{color}
\usepackage{amsthm}
\usepackage{amstext}
\usepackage{amssymb}
\usepackage{setspace}
\usepackage{esint}
\usepackage[unicode=true,pdfusetitle,
 bookmarks=true,bookmarksnumbered=false,bookmarksopen=false,
 breaklinks=false,pdfborder={0 0 0},backref=false,colorlinks=true]
 {hyperref}

\makeatletter
\numberwithin{equation}{section}
\numberwithin{figure}{section}
\theoremstyle{plain}
\newtheorem{thm}{\protect\theoremname}[section]
  \theoremstyle{definition}
  \newtheorem{defn}[thm]{\protect\definitionname}
 \newcommand\thmsname{\protect\theoremname}
 \newcommand\nm@thmtype{theorem}
 \theoremstyle{plain}
 
 \newenvironment{namedthm}[1][Undefined Theorem Name]{
   \ifx{#1}{Undefined Theorem Name}\renewcommand\nm@thmtype{theorem*}
   \else\renewcommand\thmsname{#1}\renewcommand\nm@thmtype{namedtheorem}
   \fi
   \begin{\nm@thmtype}}
   {\end{\nm@thmtype}}
  \theoremstyle{remark}
  \newtheorem{rem}[thm]{\protect\remarkname}
  \theoremstyle{remark}
  \newtheorem*{acknowledgement*}{\protect\acknowledgementname}
  \theoremstyle{plain}
  \newtheorem{lem}[thm]{\protect\lemmaname}
  \theoremstyle{plain}
  \newtheorem{cor}[thm]{\protect\corollaryname}
  \theoremstyle{plain}
  \newtheorem{prop}[thm]{\protect\propositionname}

\makeatother

  \providecommand{\acknowledgementname}{Acknowledgement}
  \providecommand{\corollaryname}{Corollary}
  \providecommand{\definitionname}{Definition}
  \providecommand{\lemmaname}{Lemma}
  \providecommand{\propositionname}{Proposition}
  \providecommand{\remarkname}{Remark}
  \providecommand{\theoremname}{Theorem}
\providecommand{\theoremname}{Theorem}

\begin{document}

\title{Floer homology for non-resonant magnetic fields on flat tori}

\author{Urs Frauenfelder, Will J. Merry and Gabriel P. Paternain}

\address{Urs Frauenfelder\\
Department of Mathematics and Research Institute of Mathematics\\
Seoul National University San 56-1 Shinrim-dong Kwanak-gu, Seoul 151-747,
Korea }

\email{\texttt{frauenf@snu.ac.kr}}

\address{Will J.~Merry\\
Department of Mathematics\\
ETH Z\"urich}

\email{\texttt{merry@math.ethz.ch}}

\address{Gabriel. P. Paternain\\
Department of Pure Mathematics and Mathematical Statistics\\
University of Cambridge, Cambridge CB3 0WB, England }

\email{\texttt{g.p.paternain@dpmms.cam.ac.uk}}

\date{\today}
\begin{abstract}
In this article we define and compute the Novikov Floer homology associated
to a non-resonant magnetic field and a mechanical Hamiltonian on a
flat torus $\mathbb{T}^{2N}$. As a result, we deduce that this Hamiltonian
system always has $2N+1$ contractible solutions, and generically
even $2^{2N}$ contractible solutions. Moreover if there exists a
non-degenerate non-contractible solution then there necessarily exists
another. 
\end{abstract}
\maketitle

\section{Introduction}

In this paper we study existence and multiplicity results for periodic
solutions of fixed period $\tau$ of a system of $N$ interacting
particles on the plane subject to a magnetic field. We suppose that
the magnetic field as well as the potential are periodic on the plane,
so that we can carry out our study on the torus. 

Let us first describe our set-up for a single particle, i.e.\,the
case $N=1$. We abbreviate by $\mathbb{T}^{2}=\mathbb{R}^{2}/\mathbb{Z}^{2}$
the two dimensional torus. We model the magnetic field via a 2-form
$\sigma\in\Omega^{2}(\mathbb{T}^{2})$. If $\pi\colon T^{*}\mathbb{T}^{2}\to\mathbb{T}^{2}$
is the footpoint projection, consider the \emph{magnetic symplectic
form}, 
\[
\omega_{\sigma}:=d\lambda+\pi^{*}\sigma
\]
where $\lambda=x_{1}dp_{1}+x_{2}dp_{2}$ is the Liouville 1-form on
the cotangent bundle of the torus. Abbreviate the circle of length
$\tau$ by $\mathbb{S}_{\tau}:=\mathbb{R}/\tau\mathbb{Z}$. For $V\in C^{\infty}(\mathbb{S}_{\tau}\times\mathbb{T}^{2},\mathbb{R})$
a $\tau$-periodic time-dependent potential we define the $\tau$-periodic
Hamiltonian $H_{V}\in C^{\infty}(\mathbb{S}_{\tau}\times T^{*}\mathbb{T}^{2},\mathbb{R})$
by 
\[
H_{V}(t,x,p)=\frac{1}{2}|p|^{2}+V(t,x)
\]
where the norm of $p$ is taken with respect to the flat metric on
the torus. The Hamiltonian vector field $X_{\sigma,V}$ of $H_{V}$
with respect to the magnetic symplectic form is implicitly defined
by the equation 
\[
\omega_{\sigma}(X_{\sigma,V},\cdot)=-dH_{V}.
\]
We are interested in $\tau$-periodic solutions of $X_{\sigma,V}$,
i.e.\,solutions $w\in C^{\infty}(\mathbb{S}_{\tau},T^{*}\mathbb{T}^{2})$
of the ODE 
\begin{equation}
\partial_{t}w=X_{\sigma,V}(t,w).\label{eq:ODE}
\end{equation}
We denote by $\mathcal{P}_{\tau}(\sigma,V)$ the moduli space of solutions.
Since the fundamental group of the two dimensional torus equals $\pi_{1}(\mathbb{T}^{2})=\pi_{1}(T^{*}\mathbb{T}^{2})=\mathbb{Z}^{2}$
we get a decomposition of this moduli space as 
\[
\mathcal{P}_{\tau}(\sigma,V)=\bigcup_{h\in\mathbb{Z}^{2}}\mathcal{P}_{\tau}^{h}(\sigma,V)
\]
where $\mathcal{P}_{\tau}^{h}(\sigma,V)$ denotes those elements in
$\mathcal{P}_{\tau}(\sigma,V)$ which represent $h$ in the fundamental
group of $T^{*}\mathbb{T}^{2}$. In particular, $\mathcal{P}_{\tau}^{0}(\sigma,V)$
denotes the contractible solutions.

We do not know if periodic solutions always exist. However, we can
prove their existence under the following non-resonance condition
we explain next. Denote by $\mu=dx_{1}\wedge dx_{2}$ the volume form
on the two torus with respect to the flat metric. Then we can write
$\sigma=a\mu$ for some function $a\in C^{\infty}(\mathbb{T}^{2},\mathbb{R})$. 
\begin{defn}
\label{nonres1} We say that $\sigma$ is \emph{non-resonant in period
$\tau$}, if there exists $k\in\mathbb{Z}$ such that 
\[
\frac{2\pi k}{\tau}<a(x)<\frac{2\pi(k+1)}{\tau},\quad\forall\,\, x\in\mathbb{T}^{2}.
\]
 
\end{defn}
Our first main result for a single particle system is the following.
\begin{namedthm}[\textbf{Theorem A}]
\textbf{}Assume that $\sigma\in\Omega^{2}(\mathbb{T}^{2})$ is non-resonant
in period $\tau$. Then for every $V\in C^{\infty}(\mathbb{S}_{\tau}\times\mathbb{T}^{2},\mathbb{R})$
we have $\#\mathcal{P}_{\tau}^{0}(\sigma,V)\geq3$. Moreover, for
a generic potential $V$ it holds that $\#\mathcal{P}_{\tau}^{0}(\sigma,V)\geq4$.\emph{ }
\end{namedthm}
In general noncontractible solutions may not necessarily exist. However,
our second main result tells us that if a noncontractible and nondegenerate
solution exists a second noncontractible one has to exist as well.
For that recall that if $\phi_{\sigma,V}^{t}$ denotes the flow of
the Hamiltonian vector field $X_{\sigma,V}$ and $w\in\mathcal{P}_{\tau}(\sigma,V)$,
then $w$ is called \emph{nondegenerate} if 
\[
\det(D\phi_{\sigma,V}^{\tau}(w(0))-\mathrm{Id})\neq0.
\]
Now we are in position to state our next main result for a single
particle system. 
\begin{namedthm}[\textbf{Theorem B}]
\textbf{}Assume that $\sigma\in\Omega^{2}(\mathbb{T}^{2})$ is a
non-resonant magnetic field in period $\tau$, $V\in C^{\infty}(\mathbb{S}_{\tau}\times\mathbb{T}^{2},\mathbb{R})$,
and $w\in\mathcal{P}_{\tau}^{h}(\sigma,V)$ is nondegenerate for some
$h\neq0\in\pi_{1}(\mathbb{T}^{2})$. Then it holds that $\#\mathcal{P}_{\tau}^{h}(\sigma,V)\geq2$. 
\end{namedthm}
We next explain how Theorem A and Theorem B generalize to systems
of $N$ interacting particles. The configuration space of $N$ particles
is given by the $2N$-dimensional torus $\mathbb{T}^{2N}$. To keep
track of our particles we choose for $1\leq j\leq N$ a torus $\mathbb{T}_{j}^{2}=\mathbb{T}^{2}$
and think of the $2N$-dimensional torus as $\mathbb{T}^{2N}=\prod_{j=1}^{N}\mathbb{T}_{j}^{2}$.
We denote by $\mathtt{p}_{j}\colon\mathbb{T}^{2N}\to\mathbb{T}_{j}^{2}$
the canonical projection. We generalize the notion of a non-resonant
magnetic field to the many particle case as follows. 
\begin{defn}
\label{nonres2}A two form $\sigma\in\Omega^{2}(\mathbb{T}^{2N})$
is called \emph{non-resonant in period $\tau$}, if for $1\leq j\leq N$
there exist non-resonant in period $\tau$ forms $\sigma_{j}\in\Omega^{2}(\mathbb{T}_{j}^{2})$
in the sense of Definition \ref{nonres1} such that $\sigma=\sum_{j=1}^{N}\mathtt{p}_{j}^{*}\sigma_{j}$. 
\end{defn}
We are now in position to generalize Theorem A to the case of $N$
particles. 
\begin{namedthm}[\textbf{Theorem A+}]
\textbf{}\textbf{\emph{ }}\emph{ }Assume that $\sigma\in\Omega^{2}(\mathbb{T}^{2N})$
is non-resonant in period $\tau$. Then for every $V\in C^{\infty}(\mathbb{S}_{\tau}\times\mathbb{T}^{2N},\mathbb{R})$
we have $\#\mathcal{P}_{\tau}^{0}(\sigma,V)\geq2N+1$. Moreover, for
a generic potential $V$ it holds that $\#\mathcal{P}_{\tau}^{0}(\sigma,V)\geq2^{2N}$ 
\end{namedthm}
Of course Theorem A+ immediately implies Theorem A by specializing
to the case $N=1$. On the other hand Theorem A+ is not a consequence
of Theorem A. Here the key point is that in Theorem A+ we do not need
to assume that the potential is of product form, which means that
our particles are allowed to interact with each other. In the non
interacting case, i.e. the case where the potential $V$ can be written
as $V=\sum_{j=1}^{N}\mathtt{p}_{j}^{*}V_{j}$ for potentials $V_{j}\in C^{\infty}(\mathbb{S}_{\tau}\times\mathbb{T}_{j}^{2},\mathbb{R})$
Theorem A+ is an immediate consequence of Theorem A. In fact, in this
case we get even the stronger lower bound $3^{N}$ for the number
of solutions.\newline

Similarly, Theorem B generalizes to the case of $N$ particles as
follows. 
\begin{namedthm}[\textbf{Theorem B+}]
\textbf{}Assume that $\sigma\in\Omega^{2}(\mathbb{T}^{2N})$ is a
non-resonant magnetic field in period $\tau$, $V\in C^{\infty}(\mathbb{S}_{\tau}\times\mathbb{T}^{2N},\mathbb{R})$,
and $w\in\mathcal{P}_{\tau}^{h}(\sigma,V)$ is nondegenerate for some
$h\neq0\in\pi_{1}(\mathbb{T}^{2N})$. Then it holds that $\#\mathcal{P}_{\tau}^{h}(\sigma,V)\geq2$.
\end{namedthm}
We conclude this Introduction by briefly explaining our approach.
Unsurprisingly, we use a variant of Floer homology. This semi-infinite
dimensional Morse homology associates to this Hamiltonian system a
chain complex $CF_{*}^{h}(\sigma,V,\tau)$ which is generated by the
elements of $\mathcal{P}_{\tau}^{h}(\sigma,V)$ and defines a boundary
operator by counting perturbed holomorphic cylinders which asymptotically
converge to the periodic orbits. A priori it is far form obvious that
this recipe gives a well defined boundary operator. Indeed, the question
as to whether Floer's boundary operator is well-defined or not depends
on a difficult compactness result for the perturbed holomorphic curve
equation, which cannot be expected to be true in full generality.
The main point of the present paper is that under the non-resonant
condition from Definition \ref{nonres1} we can establish this compactness
result. The usefulness of Floer homology stems from its invariance
under perturbations. If the magnetic field is a \emph{constant }multiple
$a_{0}\mu$ of the volume form $\mu$ (cf. Definition \ref{nonres1})
and $\tau>0$ and $k\in\{0,1,2,\dots,\}$ are such that 
\[
2\pi k<\tau\left|a_{0}\right|<2\pi(k+1),
\]
then for the special case $V\equiv0$ we can directly compute the
Floer homology $HF_{*}^{h}(a_{0}\mu,0,\tau)$: 
\begin{equation}
HF_{*}^{h}(a_{0}\mu,0,\tau)=\begin{cases}
H_{*+2k}(\mathbb{T}^{2};\mathbb{Z}), & h=0,\\
0, & h\ne0.
\end{cases}\label{eq:computation}
\end{equation}
Theorems A and B are immediate consequences of \eqref{eq:computation}
and the aforementioned invariance properties of $HF_{*}^{h}$. Strictly
speaking the case $V\equiv0$ is never non-degenerate. It is however
``weakly non-degenerate'' in the sense that the Hamiltonian action
functional is Morse-Bott. In this case it is still possible to define
the Floer homology; see Section \ref{sec:Computing-the-Novikov} for
more information. Theorems A+ and B+ follow via a similar argument.\newline

On a more technical level, there are two additional points of interest
in our construction. 
\begin{rem}
The symplectic form $\omega_{\sigma}$ is \emph{not atoroidal}. This
means that the Hamiltonian action functional (cf. \eqref{eq:HAF})
is not real valued on the loop space of $T^{*}\mathbb{T}^{n}$, and
one needs to use a suitable \emph{Novikov }cover. Thus the Floer homology
$HF_{*}^{h}$ that we construct is actually a Novikov Floer homology.
\end{rem}

\begin{rem}
\label{lagrangian remark}As is well known, an alternative (more classical)
approach to proving results like Theorems A and B is\emph{ }the \emph{Lagrangian
formulation }where one studies a Lagrangian action functional on the
(completed) loop space of $\mathbb{T}^{2}$, and attempts to define
a Morse homology for it. When defined, this Morse homology should
agree with the Floer homology. For the flavours of Morse/Floer homology
we use in this paper this result is due to Abbondandolo and Schwarz
\cite{AbbondandoloSchwarz2006} and is explained in more detail in
this setting in \cite[Appendix A]{FrauenfelderMerryPaternain2012}.
In Section \ref{sec:The-Lagrangian-Setting} we explain how the non-resonance
condition from Definition \ref{nonres1} implies that the Lagrangian
action functional satisfies the \emph{Palais-Smale condition}. If
$\sigma$ is sufficiently small (more precisely, if one can take $k=0$
in Definition \ref{nonres1}) then the Lagrangian action functional
is also bounded below, and in this case one can recover Theorems A
and B using this functional (see e.g. \cite{AbbondandoloMajer2006}).
However if $k>0$ in Definition \ref{nonres1} then the Lagrangian
action functional is no longer bounded below, and hence one cannot
define a Morse homology with it. In fact, \eqref{eq:computation}
shows that for $k>0$, the Floer homology is not the homology of a
topological space (it is zero in degree zero) and hence one should
not expect to be able to define a Morse homology for the Lagrangian
action functional. 
\end{rem}
The results of this paper were announced in our earlier article \cite{FrauenfelderMerryPaternain2012}.
\begin{acknowledgement*}
We are grateful to Felix Schlenk for his helpful comments and discussions.
UF is supported by the Alexander von Humboldt Foundation and by the
Basic Research fund 2010-0007669 funded by the Korean government.
WM is supported by an ETH Postdoctoral Fellowship. 
\end{acknowledgement*}

\section{Preliminaries}

We think of $\mathbb{T}^{2N}$ as being embedded inside $\mathbb{C}^{2N}$
as $S^{1}\times\dots\times S^{1}$. We denote by $\left\langle \cdot,\cdot\right\rangle $
the standard Euclidean flat metric on $\mathbb{T}^{2N}$, which comes
from the real part of the Hermitian inner product on $\mathbb{C}^{2N}$,
and we denote by 
\[
\mathtt{j}=\left(\begin{array}{cc}
 & -\mathrm{Id}\\
\mathrm{Id}
\end{array}\right)
\]
the canonical almost complex structure on $\mathbb{C}^{2N}$. 

The embedding $\mathbb{T}^{2N}\subset\mathbb{C}^{2N}$ induces an
embedding of the trivial bundle $T^{*}\mathbb{T}^{2N}$ inside $\mathbb{C}^{4N}$.
To minimize ambiguity, we write $\left\langle \left\langle \cdot,\cdot\right\rangle \right\rangle $
for the induced metric on $T^{*}\mathbb{T}^{2N}$ and $\mathtt{J}$
for the corresponding almost complex structure on $T^{*}\mathbb{T}^{2N}$
(thus $\left\langle \left\langle \cdot,\cdot\right\rangle \right\rangle $
and $\mathtt{J}$ are defined in the same way as $\left\langle \cdot,\cdot\right\rangle $
and $\mathtt{j}$, but on $\mathbb{C}^{4N}$ instead of $\mathbb{C}^{2N}$).
We will freely use the ``musical'' isomorphism $v\mapsto\left\langle v,\cdot\right\rangle $
to identify $T\mathbb{T}^{2N}$ with $T^{*}\mathbb{T}^{2N}$ without
further comment. Under this identification $\left\langle \left\langle \cdot,\cdot\right\rangle \right\rangle $
is the \emph{Sasaki }metric on $T\mathbb{T}^{2N}$. 

In addition to the isometry $T\mathbb{T}^{2N}\cong T^{*}\mathbb{T}^{2N}$,
the metric also determines a \emph{horizontal-vertical splitting }of
the tangent bundle $TT^{*}\mathbb{T}^{2N}$:
\begin{equation}
T_{(x,p)}T^{*}\mathbb{T}^{2N}\cong T_{x}\mathbb{T}^{2N}\oplus T_{x}^{*}\mathbb{T}^{2N}\label{eq:hv splitting}
\end{equation}
which sends a tangent vector $\xi$ to the pair $\xi\mapsto(D\pi(\xi),K(\xi))$;
here $K$ is the connection map associated to the Levi-Civita connection
$\nabla$ induced from $\left\langle \cdot,\cdot\right\rangle $.
In fact, under our embedding this corresponds to the splitting $\mathbb{C}^{4N}\cong\mathbb{C}^{2N}\oplus\mathbb{C}^{2N}$.
In particular, if we write $\xi^{h}:=D\pi(\xi)$ and $\xi^{v}:=K(\xi)$
then one has
\[
\left\langle \left\langle \xi,\zeta\right\rangle \right\rangle =\left\langle \xi^{h},\zeta^{h}\right\rangle +\left\langle \xi^{v},\zeta^{v}\right\rangle .
\]
Using this notation, the almost complex structure $\mathtt{J}$ sends
$(\xi^{h},\xi^{v})\mapsto(-\xi^{v},\xi^{h})$. The almost complex
structure $\mathtt{J}$ is \emph{compatible} with the standard symplectic
structure $d\lambda$ in the sense that 
\[
d\lambda(\mathtt{J}\cdot,\cdot)=\left\langle \left\langle \cdot,\cdot\right\rangle \right\rangle .
\]

Now let us introduce a closed 2-form $\sigma$ on $\mathbb{T}^{2N}$.
Whilst all of what follows in this section is valid for any closed
2-form $\sigma$, since in this article we will only work with product
2-forms $\sigma=\sum_{j=1}^{N}\mathtt{p}_{j}^{*}\sigma_{j}$, where
each $\sigma_{j}\in\Omega^{2}(\mathbb{T}^{2})$, we shall restrict
to this case right away. Such a 2-form $\sigma$ is equivalent to
an $N$-tuple $(a_{1},\dots,a_{N})$ of smooth functions $a_{j}:\mathbb{T}^{2}\rightarrow\mathbb{R}$,
where $\sigma_{j}=a_{j}\mu$. Here $\mu$ is the volume form on $\mathbb{T}^{2}$
induced by $\left\langle \cdot,\cdot\right\rangle $. 

The metric $\left\langle \cdot,\cdot\right\rangle $ allows us to
associate to $\sigma$ an endomorphism $Y_{\sigma}$ of $T\mathbb{T}^{2N}$,
called the \emph{Lorentz force}, via 
\[
\sigma_{x}(v,v')=\left\langle Y_{\sigma}(x)\cdot v,v'\right\rangle .
\]

We denote by $\mathcal{J}_{\sigma}$ the open set of almost complex
structures $J$ on $T^{*}\mathbb{T}^{2N}$ that are both \emph{compatible
}with $d\lambda$ and \emph{tamed}\textbf{ }by $\omega_{\sigma}$
- this just means that the bilinear form $d\lambda(J.,\cdot)$ is
both positive definite and symmetric, and the bilinear form $\omega_{\sigma}(J\cdot,\cdot)$
is positive definite (but not necessarily symmetric). Unfortunately
in general $\mathtt{J}$ will \emph{not }belong to $\mathcal{J}_{\sigma}$.
This can be rectified by rescaling. Given $A>0$ we define the rescaled
almost complex structure
\[
\mathtt{J}_{A}:=\left(\begin{array}{cc}
 & -\tfrac{1}{A}\cdot\mathrm{Id}\\
A\cdot\mathrm{Id}
\end{array}\right).
\]

\begin{defn}
\label{unitame def}Fix $A,\varepsilon>0$. Let $\mathcal{U}(\sigma,A,\varepsilon)$
denote the set of almost complex structures $J\in\mathcal{J}_{\sigma}$
which are \emph{uniformly tame }in the sense that
\begin{equation}
\omega_{\sigma}(J\xi,\xi)>\frac{1}{4}d\lambda(J\xi,\xi)\ \ \ \mbox{for all }\xi\in TT^{*}\mathbb{T}^{2N},\label{eq:second}
\end{equation}
and which satisfy 
\[
\left\Vert J-\mathtt{J}_{A}\right\Vert _{\infty}<\varepsilon.
\]
A priori there is no reason why $\mathcal{U}(\sigma,A,\varepsilon)$
should be non-empty. However in \cite[Lemma 2]{FrauenfelderMerryPaternain2012},
we proved:\end{defn}
\begin{lem}
\label{lem:uniformly tame}If $\sigma=\sum_{j=1}^{N}\mathtt{p}_{j}^{*}(a_{j}\mu)$
and $A\geq\max_{1\leq j\leq N}\left\Vert a_{j}\right\Vert _{\infty}$
then for any $0<\varepsilon<\tfrac{1}{7}$ the set $\mathcal{U}(\sigma,A,\varepsilon)$
is open and non-empty. More precisely, one has 
\[
B(\mathtt{J}_{A},\varepsilon)\subset\mathcal{U}(\sigma,A,\varepsilon),
\]
where $B(\mathtt{J}_{A},\varepsilon):=\left\{ J\in\mathcal{J}_{\sigma}\mid\left\Vert J-\mathtt{J}_{A}\right\Vert _{\infty}<\varepsilon\right\} $.
\end{lem}
We will also need the following easy result.
\begin{lem}
\label{lem:AS bound}Suppose $J$ is an almost complex structure on
$T^{*}\mathbb{T}^{2N}$ which is compatible with $d\lambda$. Let
$\kappa_{J}$ denote the minimal positive eigenvalue of the positive
symmetric operator $-\mathtt{J}\circ J$. Then for any $v\in TT^{*}\mathbb{T}^{2N}$
one has 
\[
\left\langle \left\langle v,v\right\rangle \right\rangle \leq\tfrac{1}{\kappa_{J}}d\lambda(Jv,v).
\]
\end{lem}
\begin{proof}
A simple computation:
\begin{align*}
d\lambda(Jv,v) & =d\lambda(\mathtt{J}Jv,\mathtt{J}v)\\
 & =\left\langle \left\langle Jv,\mathtt{J}v\right\rangle \right\rangle \\
 & =\left\langle \left\langle -\mathtt{J}\circ Jv,v\right\rangle \right\rangle \\
 & \geq\kappa_{J}\left\langle \left\langle v,v\right\rangle \right\rangle .
\end{align*}

\end{proof}
Given $\tau>0$ we write $\mathcal{L}_{\tau}\mathbb{T}^{2N}=C^{\infty}(\mathbb{S}_{\tau},\mathbb{T}^{2N})$.
For each $h\in\mathbb{Z}^{2N}=\pi_{1}(\mathbb{T}^{2N})$ we write
$\mathcal{L}_{\tau}^{h}\mathbb{T}^{2N}$ for the component of $\mathcal{L}_{\tau}\mathbb{T}^{2N}$
of loops belonging to the homotopy class $h$. We write $\mathcal{L}_{\tau}T^{*}\mathbb{T}^{2N}$
for the free loop space of the cotangent bundle. We define the $L^{2}$-inner
products $\left\langle \cdot,\cdot\right\rangle _{2}$ on $\mathcal{L}_{\tau}\mathbb{T}^{2N}$
and $\left\langle \left\langle \cdot.\cdot\right\rangle \right\rangle _{2}$
on $\mathcal{L}_{\tau}T^{*}\mathbb{T}^{2N}$ using the metrics $\left\langle \cdot,\cdot\right\rangle $
and $\left\langle \left\langle \cdot,\cdot\right\rangle \right\rangle $
respectively. Thus 
\[
\left\langle \left\langle \xi,\zeta\right\rangle \right\rangle _{2}:=\int_{0}^{\tau}\left\langle \left\langle \xi(t),\zeta(t)\right\rangle \right\rangle dt.
\]
Let us now rephrase the non-resonance condition in a way that will
be more convenient later on. Given $\gamma\in\mathcal{L}_{\tau}\mathbb{T}^{2N}$,
we denote by $F_{\sigma}^{\gamma}\in\Gamma(\gamma^{*}\mbox{End}(T\mathbb{T}^{2N}))$
the unique solution to the Cauchy problem 
\begin{equation}
\nabla_{t}F_{\sigma}^{\gamma}(t)=F_{\sigma}^{\gamma}(t)\circ Y_{\sigma}(\gamma(t)),\ \ \ F_{\sigma}^{\gamma}(0)=\mathrm{Id}.\label{eq:F sigma}
\end{equation}
Since $Y_{\sigma}$ is antisymmetric, the operator $F_{\sigma}^{\gamma}$
is an orthogonal transformation. If $\sigma=\sum_{j=1}^{N}\mathtt{p}_{j}^{*}(a_{j}\mu)$
and 
\[
b_{j}^{\gamma}(t):=\int_{0}^{t}a_{j}(\mathtt{p}_{j}(\gamma(s)))ds
\]
then one has for $(v_{1},\dots,v_{2N})\in T\mathbb{T}^{2N}$ that
\[
F_{\sigma}^{\gamma}(t)(v_{1},\dots,v_{2N})=\left(\exp(b_{1}^{\gamma}(t)\mathtt{j})(v_{1},v_{2}),\dots,\exp(b_{N}^{\gamma}(t)\mathtt{j})(v_{2N-1},v_{2N})\right)
\]
where here $\mathtt{j}$ is the complex structure on $\mathbb{T}^{2}$
(not $\mathbb{T}^{2N}$), and each pair $(v_{2j-1},v_{2j})$ is thought
of as belonging to $T\mathbb{T}^{2}$. The following lemma is therefore
straightforward.
\begin{lem}
\label{lem:reform of nonres}Suppose $\sigma=\sum_{j=1}^{N}\mathtt{p}_{j}^{*}\sigma_{j}$.
Then $\sigma$ is non-resonant in period $\tau$ (in the sense of
Definitions \ref{nonres1} and \ref{nonres2}) if and only if there
exists $\varepsilon>0$ such that for any curve $\gamma\in\mathcal{L}_{\tau}\mathbb{T}^{2N}$
and all $v\in T_{\gamma(\tau)}\mathbb{T}^{2N}$, one has 
\begin{equation}
\left|(F_{\sigma}^{\gamma}(\tau)-\mathrm{Id})v\right|\geq\varepsilon\left|v\right|.\label{eq:alt nonres}
\end{equation}
\end{lem}
\begin{rem}
\label{more general nonres}Of course, one can still define $F_{\sigma}^{\gamma}$
even when $\sigma$ is not of the form $\sum_{j=1}^{N}\mathtt{p}_{j}^{*}\sigma_{j}$.
The construction in Section \ref{sec:Defining-the-Novikov} would
go through without any changes if we took \eqref{eq:alt nonres} as
the definition of non-resonance rather than Definitions \ref{nonres1}
and \ref{nonres2}, which thus allows us to define the Floer homology
groups $HF_{*}^{h}(\sigma,V,\tau)$ for magnetic forms that are not
of the form $\sigma=\sum_{j=1}^{N}\mathtt{p}_{j}^{*}\sigma_{j}$.
Nevertheless, we do not know how to calculate $HF_{*}^{h}(\sigma,V,\tau)$
unless $\sigma$ is of this form, which is why Theorem A+ and Theorem
B+ are stated for product magnetic forms. 
\end{rem}

\section{\label{sec:Defining-the-Novikov}Defining the Novikov Floer homology}

\subsection{The Novikov framework}

Consider the 1-form $a_{\sigma}\in\Omega^{1}(\mathcal{L}_{\tau}\mathbb{T}^{2N})$
defined by 
\begin{equation}
a_{\sigma}(\gamma)(\xi):=\int_{\mathbb{S}_{\tau}}\sigma(\partial_{t}\gamma,\xi)dt.\label{eq:a sigma}
\end{equation}
 Since $\sigma$ is closed, $a_{\sigma}$ is closed, and thus for
each $h\in\mathbb{Z}^{2N}\cong\pi_{1}(\mathbb{T}^{2N})$ we can define
a map $\Phi_{\sigma}^{h}:\pi_{1}(\mathcal{L}_{\tau}^{h}\mathbb{T}^{2N})\rightarrow\mathbb{R}$
by 
\[
\Phi_{\sigma}^{\alpha}([f]):=\int_{S^{1}}f^{*}a_{\sigma}=\int_{S^{1}\times\mathbb{S}_{\tau}}f^{*}\sigma,
\]
 where we think of a map $f:S^{1}\rightarrow\mathcal{L}_{\tau}^{h}\mathbb{T}^{2N}$
representing an element $[f]\in\pi_{1}(\mathcal{L}_{\tau}^{h}\mathbb{T}^{2N})$
also as a map $f:S^{1}\times\mathbb{S}_{\tau}\rightarrow\mathbb{T}^{2N}$. 

Given $h\in\mathbb{Z}^{2N}$, thought of as a lattice in $\mathbb{R}^{2N}$,
let $\widetilde{\gamma}_{h}:[0,1]\rightarrow\mathbb{R}^{2N}$ be defined
by 
\begin{equation}
\widetilde{\gamma}_{h}(t):=th.\label{eq:reference loop}
\end{equation}
Let $\gamma_{h}\in\mathcal{L}_{1}^{h}\mathbb{T}^{2N}$ denote the
projection $\gamma_{h}=q\circ\widetilde{\gamma}_{h}$, where $\mathtt{q}:\mathbb{R}^{2N}\rightarrow\mathbb{T}^{2N}$
denotes the universal cover. Denote by 
\[
\Gamma_{\sigma}^{h}:=\pi_{1}(\mathcal{L}_{\tau}^{h}\mathbb{T}^{2N})/\ker(\Phi_{\sigma}^{h}).
\]
Then $\Gamma_{\sigma}^{h}$ is a finitely generated free abelian group,
$\Gamma_{\sigma}^{h}\cong\mathbb{Z}^{m}$ for some $m$. $\Gamma_{\sigma}^{h}$
is cyclic if and only if $[a_{\sigma}]\in H^{1}(\mathcal{L}_{\tau}\mathbb{T}^{2N};\mathbb{Z})$
is an integral class (see \cite[Lemma 2.1]{Farber2004}), that is,
if and only if $\sigma\in H^{2}(\mathbb{T}^{2N};\mathbb{Z})$ is integral.
Let $Q:\mathbb{L}_{\tau}^{h}(\mathbb{T}^{2N},\sigma)\rightarrow\mathcal{L}_{\tau}^{h}\mathbb{T}^{2N}$
denote the covering space of $\mathcal{L}_{\tau}^{h}\mathbb{T}^{2N}$
with deck transformation group $\Gamma_{\sigma}^{h}$. We call $Q$
a \emph{finite integration cover for}\textbf{ $\sigma$ }of $\mathcal{L}_{\tau}^{h}\mathbb{T}^{2N}$.
An element of $\mathbb{L}_{\tau}^{h}(\mathbb{T}^{2N},\sigma)$ is
an equivalence class $[\gamma,z]$ of a pair $(\gamma,z)$, where
$\gamma\in\mathcal{L}_{\tau}^{h}\mathbb{T}^{2N}$ and $z:[0,1]\times\mathbb{S}_{\tau}\rightarrow\mathbb{T}^{2N}$
satisfies 
\begin{equation}
z(0,t)=\gamma_{h}(t/\tau),\ \ \ z(1,t)=\gamma(t),\label{eq:bdy cdn-1}
\end{equation}
and the equivalence relation is given by 
\begin{equation}
(\gamma,z_{0})\sim(\gamma,z_{1})\ \ \ \Leftrightarrow\ \ \ [z_{0}\sharp z_{1}^{-}]\in\ker(\Phi_{\sigma}^{h}),\label{eq:equiv relation}
\end{equation}
where $z_{1}^{-}(s,t):=z_{1}(s,-t)$, and $[z_{0}\sharp z_{1}^{-}]$
denotes the element of $\pi_{1}(\mathcal{L}_{\tau}^{h}\mathbb{T}^{2N})$
containing the map $S^{1}\times\mathbb{S}_{\tau}\rightarrow\mathbb{T}^{n}$
obtained by gluing $z_{0}$ and $z_{1}^{-}$ along $\gamma_{h}$. 
\begin{rem}
\label{weakly exact remark}Note for $h=0$ one has $\Phi_{\sigma}^{0}\equiv0$
since the lift of $\sigma$ to the universal cover $\mathbb{R}^{2N}$
of $\mathbb{T}^{2N}$ is necessarily exact, and hence $\int_{S^{2}}f^{*}\sigma=0$
for any map $f:S^{2}\rightarrow\mathbb{T}^{2N}$. Thus for $h=0$
the space $\mathbb{L}_{\tau}^{0}(\mathbb{T}^{2N},\sigma)$ is exactly
$\mathcal{L}_{\tau}^{0}\mathbb{T}^{2N}$.
\end{rem}
The action of $\Gamma_{\sigma}^{h}$ on the fibre $Q^{-1}(\gamma)$
is given by $(f,[\gamma,z])\mapsto[\gamma,z\sharp f]$, where $f:S^{1}\times\mathbb{S}_{\tau}\rightarrow\mathbb{T}^{2N}$
is a representative of $f\in\Gamma_{\sigma}^{h}$, and $z\sharp f:[0,1]\times\mathbb{S}_{\tau}\rightarrow\mathbb{T}^{2N}$
is (a smooth reparametrization of) the map 
\begin{equation}
(z\sharp f)(s,t):=\begin{cases}
z(2r,t), & 0\leq r\leq1/2,\\
f(2r-1,t), & 1/2\leq r\leq1.
\end{cases}\label{eq:action of fibre}
\end{equation}
It follows directly from the definition that the one-form $Q^{*}a_{\sigma}\in\Omega^{1}(\mathbb{L}_{\tau}^{h}(\mathbb{T}^{2N},\sigma))$
is exact. In fact, if $\mathcal{A}_{\sigma}:\mathbb{L}_{\tau}^{h}(\mathbb{T}^{2N},\sigma)\rightarrow\mathbb{R}$
is defined by 
\begin{equation}
\mathcal{A}_{\sigma}([\gamma,z]):=\int_{[0,1]\times\mathbb{S}_{\tau}}z^{*}\sigma,\label{eq:A sigma}
\end{equation}
then we have 
\[
Q^{*}a_{\sigma}=d\mathcal{A}_{\sigma}.
\]
Note that 
\[
\mathcal{A}_{\sigma}([\gamma,z\#f])=\mathcal{A}_{\sigma}([\gamma,z])+\Phi_{\sigma}^{h}(f).
\]
Let $\Lambda_{\sigma}^{h}$ denote the\textbf{\emph{ $\mathbb{Z}_{2}$}}\emph{-Novikov
ring} 
\begin{equation}
\Lambda_{\sigma}^{h}=\left\{ \sum_{k=0}^{\infty}c_{k}f_{k}\mid c_{k}\in\mathbb{Z}_{2},\ f_{k}\in\Gamma_{\sigma}^{h},\ \Phi_{\sigma}^{h}(f_{k})\rightarrow-\infty\mbox{ as }k\rightarrow\infty\right\} ;\label{eq:novikov ring}
\end{equation}
The fact that $\Gamma_{\sigma}^{h}$ is a finitely generated free
abelian group implies that $\Lambda_{\sigma}^{h}$ is always a field
(\cite[Theorem 4.1]{HoferSalamon1995}). 
\begin{rem}
\label{projective line}Note that $\mathbb{L}_{\tau}^{h}(\mathbb{T}^{2N},\sigma)$
and $\Lambda_{\sigma}^{h}$ depend only on the projective line in
$H^{2}(\mathbb{T}^{2N};\mathbb{R})$ determined by $\sigma$, that
is, on $\{t[\sigma]\,:\, t\in\mathbb{R}\backslash\{0\}\}\subset H^{2}(\mathbb{T}^{2N};\mathbb{R})$.
Indeed, it suffices to observe that $\ker(\Phi_{t\sigma+d\theta}^{h})=\ker(\Phi_{\sigma}^{h})$
for each $t\in\mathbb{R}\backslash\{0\}$ and each $\theta\in\Omega^{1}(\mathbb{T}^{2N})$. 
\end{rem}
It will be convenient to view these covering spaces also as covering
spaces of the free loop space $\mathcal{L}_{\tau}^{h}T^{*}\mathbb{T}^{2N}$
of $T^{*}\mathbb{T}^{2N}$. To this end, given $h\in\mathbb{Z}^{2N}$,
we denote by $\widehat{Q}:\mathbb{L}_{\tau}^{h}(T^{*}\mathbb{T}^{2N},\sigma)\rightarrow\mathcal{L}_{\tau}^{h}T^{*}\mathbb{T}^{2N}$
the cover whose fibre $\widehat{Q}^{-1}(w)$ over $w\in\mathcal{L}_{\tau}^{h}T^{*}\mathbb{T}^{2N}$
is simply the fibre $Q^{-1}(\pi\circ w)$ for the cover $Q:\mathbb{L}_{\tau}^{h}(\mathbb{T}^{2N},\sigma)\rightarrow\mathcal{L}_{\tau}^{h}\mathbb{T}^{2N}$
defined above. In other words, $\mathbb{L}_{\tau}^{h}(T^{*}\mathbb{T}^{2N},\sigma)$
consists of equivalence classes $[w,z]$ of pairs $(w,z)$, where
$w\in\mathcal{L}_{\tau}^{h}T^{*}\mathbb{T}^{n}$ and $z$ satisfies
\eqref{eq:bdy cdn-1} (with $\gamma$ replaced by $\pi\circ w$),
and the equivalence relation is the same as \eqref{eq:equiv relation}.
Note that the group of deck transformations of $\mathbb{L}_{\tau}^{h}(T^{*}\mathbb{T}^{2N},\sigma)$
is again $\Gamma_{\sigma}^{h}$, and the action $(f,[w,z])\rightarrow[w,z\sharp f]$
of $\Gamma_{\sigma}^{f}$ on the fibre $\widehat{Q}^{-1}(w)$ is the
same as in \eqref{eq:action of fibre}.

Next, given a smooth potential $V\in C^{\infty}(\mathbb{S}_{\tau}\times\mathbb{T}^{2N},\mathbb{R})$,
let $H_{V}\in C^{\infty}(\mathbb{S}_{\tau}\times T^{*}\mathbb{T}^{2N},\mathbb{R})$
be defined by 
\[
H_{V}(t,x,p)=\frac{1}{2}\left|p\right|^{2}+V(t,x).
\]
We define the \emph{Hamiltonian action functional} $\mathcal{A}_{V}:\mathcal{L}_{\tau}^{h}T^{*}\mathbb{T}^{2N}\rightarrow\mathbb{R}$
by
\begin{equation}
\mathcal{A}_{V}(w):=\int_{\mathbb{S}_{\tau}}w^{*}\lambda-\int_{0}^{\tau}H_{V}(t,w)dt.\label{eq:HAF}
\end{equation}
It is well known that the set of critical points of $\mathcal{A}_{V}$
corresponds precisely to the set $\mathcal{P}_{\tau}(\sigma=0,V)$
of solutions to the ODE \eqref{eq:ODE} (with $\sigma=0$). Fix $h\in\mathbb{Z}^{2N}$,
and define 
\[
\mathcal{A}_{\sigma,V}:\mathbb{L}_{\tau}^{h}(T^{*}\mathbb{T}^{2N},\sigma)\rightarrow\mathbb{R};
\]
by
\begin{align*}
\mathcal{A}_{\sigma,V}([w,z]) & :=\mathcal{A}_{V}(\widehat{Q}([w,z]))+\mathcal{A}_{\sigma}([\pi\circ w,z])\\
 & =\int_{\mathbb{S}_{\tau}}w^{*}\lambda-\int_{0}^{\tau}H_{V}(t,w)dt+\int_{[0,1]\times\mathbb{S}_{\tau}}z^{*}\sigma.
\end{align*}
Then 
\[
\mathcal{A}_{\sigma,V}([w,z\sharp f])=\mathcal{A}_{\sigma,V}([w,z])+\Phi_{\sigma}^{h}(f),\ \ \ \mbox{for }f\in\Gamma_{\sigma}^{h},
\]
and it is not hard to check that the set $\mbox{Crit}_{\tau}^{h}(\mathcal{A}_{\sigma,V})$
of critical points of $\mathcal{A}_{\sigma,V}$ on $\mathbb{L}_{\tau}^{h}(T^{*}\mathbb{T}^{2N},\sigma)$
is simply the preimage $\widehat{Q}^{-1}(\mathcal{P}_{\tau}^{h}(\sigma,V))$:
\[
\mbox{Crit}_{\tau}^{h}(\mathcal{A}_{\sigma,V})=\left\{ [w,z]\in\mathbb{L}_{\tau}^{h}(\mathbb{T}^{2N},\sigma)\mid w\in\mathcal{P}_{\tau}^{h}(\sigma,V)\right\} .
\]
Recall that a critical point $[w,z]$ of $\mathcal{A}_{\sigma,V}$
is said to be \emph{non-degenerate }if 
\begin{equation}
\det(D\phi_{\sigma,V}^{\tau}(w(0))-\mathrm{Id})\neq0,\label{eq:non deg eq}
\end{equation}
where $\phi_{\sigma,V}^{t}$ denotes the flow of $X_{\sigma,V}$. 
\begin{defn}
We say that the triple $(\sigma,V,\tau)$ is \emph{non-degenerate
}if every element of $\mbox{Crit}_{\tau}^{h}(\mathcal{A}_{\sigma,V})$
is non-degenerate (for all $h\in\mathbb{Z}^{2N})$. Equivalently,
$(\sigma,V,\tau)$ is non-degenerate if and only if $\mathcal{A}_{\sigma,V}:\mathbb{L}_{\tau}^{h}(T^{*}\mathbb{T}^{2N},\sigma)\rightarrow\mathbb{R}$
is a Morse function for all $h\in\mathbb{Z}^{2N}$.
\end{defn}
The next result is very standard; a proof for $\sigma=0$ can be found
in \cite[Theorem 1.1]{Weber2002}, and the same argument goes through
with only minor changes in the general case.
\begin{thm}
\label{thm:transversality}For a given closed 2-form $\sigma\in\Omega^{2}(\mathbb{T}^{2N})$
and a given $\tau>0$, the set of potentials $V\in C^{\infty}(\mathbb{S}_{\tau}\times\mathbb{T}^{2N},\mathbb{R})$
for which $(\sigma,V,\tau)$ is non-degenerate, is of second category
in $C^{\infty}(\mathbb{S}_{\tau}\times\mathbb{T}^{2N},\mathbb{R})$. 
\end{thm}
Set 
\[
\mathcal{F}_{\tau}^{h}(\sigma,V;\delta):=\left\{ w\in\mathcal{L}_{\tau}^{h}T^{*}\mathbb{T}^{2N}\mid\left\Vert \partial_{t}w-X_{\sigma,V}(t,w)\right\Vert _{2}\leq\delta\right\} .
\]
The following lemma explains why the non-resonance condition from
Definition \ref{nonres1} is important. 
\begin{lem}
\label{lem:condition}Assume $\sigma$ is non-resonant in period $\tau$,
and let $\varepsilon>0$ be such that \eqref{eq:alt nonres} holds.
Then for all $w=(\gamma,p)\in\mathcal{F}_{\tau}^{h}(\sigma,V;\delta)$
one has
\begin{equation}
\left\Vert p\right\Vert _{\infty}\leq\left(\sqrt{\tau}+\frac{\sqrt{2\tau}}{\varepsilon}\right)(\delta+\sqrt{\tau}\left\Vert \nabla V\right\Vert _{\infty}).\label{eq:linftiy bound}
\end{equation}
\end{lem}
\begin{proof}
In terms of the horizontal vertical splitting \eqref{eq:hv splitting}
one has 
\begin{align}
\partial_{t}w-X_{\sigma,V}(t,w) & =\left(\begin{array}{c}
\partial_{t}\gamma-p\\
\nabla_{t}p+Y_{\sigma}(\gamma)p+\nabla V_{t}(\gamma)
\end{array}\right)\label{eq:the symp gradient}
\end{align}
and hence 
\[
\left\Vert \partial_{t}w-X_{\sigma,V}(t,w)\right\Vert _{2}^{2}=\int_{0}^{\tau}\left|\nabla_{t}p+Y_{\sigma}(\gamma)p+\nabla V_{t}(\gamma)\right|^{2}dt+\int_{0}^{\tau}\left|\partial_{t}\gamma-p\right|^{2}dt.
\]
Fix $w\in\mathcal{F}_{\tau}^{h}(\sigma,V;\delta)$. Since $T^{*}\mathbb{T}^{2N}\cong\mathbb{T}^{2N}\times\mathbb{R}^{2N}$
is a trivial vector bundle, it makes sense to consider the path $y=(\gamma(0),\eta):[0,\tau]\rightarrow T^{*}\mathbb{T}^{2N}$
by 
\[
\eta(t):=F_{\sigma}^{\gamma}(t)\cdot p(t),
\]
where $F_{\sigma}^{\gamma}(t)$ was defined in \eqref{eq:F sigma}.
Observe that 
\begin{align*}
\nabla_{t}\eta & =(\nabla_{t}F_{\sigma}^{\gamma})\cdot p+F_{\sigma}^{\gamma}\cdot\nabla_{t}p\\
 & =F_{\sigma}^{\gamma}(\nabla_{t}p+Y_{\sigma}(\gamma)p),
\end{align*}
Since $F_{\sigma}^{\gamma}(t)$ is an orthogonal transformation of
$T_{\gamma(t)}\mathbb{T}^{2N}$, we have
\begin{align*}
\int_{0}^{\tau}\left|\nabla_{t}\eta(t)\right|^{2}dt & =\int_{0}^{\tau}\left|\nabla_{t}p+Y_{\sigma}(\gamma)p\right|^{2}dt\\
 & \leq2\left(\int_{0}^{\tau}\left|\nabla_{t}p+Y_{\sigma}(\gamma)p+\nabla V_{t}(\gamma)\right|^{2}dt+\int_{0}^{\tau}\left|\nabla V_{t}(\gamma)\right|^{2}dt\right)\\
 & \leq2\delta^{2}+2\tau\left\Vert \nabla V\right\Vert _{\infty}^{2},
\end{align*}
where we used $a^{2}\leq2(a+b)^{2}+2b^{2}$. Thus in particular, 
\begin{equation}
\mbox{dist}(y(0),y(\tau))\leq\int_{0}^{\tau}\left|\nabla_{t}\eta(t)\right|dt\leq\sqrt{\tau}\cdot\sqrt{2\delta^{2}+2\tau\left\Vert \nabla V\right\Vert _{\infty}^{2}},\label{eq:lower}
\end{equation}
where $\mbox{dist}$ is the distance measured with with respect to
the metric $\left\langle \left\langle \cdot,\cdot\right\rangle \right\rangle $
on $T\mathbb{T}^{2N}$. However since 
\[
\eta(0)=p(0),\ \ \ \eta(\tau)=F_{\sigma}^{\gamma}(\tau)\cdot p(\tau),
\]
one has 
\begin{equation}
\mbox{dist}(y(0),y(\tau))\geq\left|(F_{\sigma}^{\gamma}(\tau)-\mathrm{Id})p(0)\right|\geq\varepsilon\left|p(0)\right|,\label{eq:upper}
\end{equation}
where we are using the fact that $\left\langle \cdot,\cdot\right\rangle $
is flat to conclude that the $\left\langle \left\langle \cdot,\cdot\right\rangle \right\rangle $-geodesic
running from $y(0)=(\gamma(0),\eta(0))$ to $y(\tau)=(\gamma(0),\eta(\tau))$
is the straight line in $T_{\gamma(0)}\mathbb{T}^{n}$ from $\eta(0)$
to $\eta(\tau)$. Combining \eqref{eq:lower} and \eqref{eq:upper}
we conclude that 
\[
\left|p(0)\right|\leq\frac{\sqrt{\tau}}{\varepsilon}\sqrt{2\delta^{2}+2\tau\left\Vert \nabla V\right\Vert _{\infty}^{2}}\leq\frac{\sqrt{2\tau}}{\varepsilon}(\delta+\sqrt{\tau}\left\Vert \nabla V\right\Vert _{\infty}).
\]
Next, given $s\in\mathbb{S}_{\tau}$ one has 
\begin{align*}
\left|\left|p(s)\right|-\left|p(0)\right|\right| & \leq\int_{0}^{s}\left|\nabla_{t}(\left|p(t)\right|)\right|dt\\
 & \leq\int_{0}^{\tau}\frac{\left|\left\langle \nabla_{t}p,p\right\rangle \right|}{\left|p\right|}dt\\
 & \overset{(*)}{=}\int_{0}^{\tau}\frac{\left|\left\langle \nabla_{t}p+Y_{\sigma}(\gamma)p,p\right\rangle \right|}{\left|p\right|}dt\\
 & \leq\int_{0}^{\tau}\left|\nabla_{t}p+Y_{\sigma}(\gamma)p\right|dt\\
 & \leq\int_{0}^{\tau}\left|\nabla_{t}p+Y_{\sigma}(\gamma)p+\nabla V_{t}(\gamma)\right|dt+\int_{0}^{\tau}\left|\nabla V_{t}(\gamma)\right|dt\\
 & \leq\sqrt{\tau}\left(\int_{0}^{\tau}\left|\nabla_{t}p+Y_{\sigma}(\gamma)p+\nabla V_{t}(\gamma)\right|^{2}dt\right)^{1/2}+\tau\left\Vert \nabla V\right\Vert _{\infty}\\
 & \leq\sqrt{\tau}\left(\delta+\sqrt{\tau}\left\Vert \nabla V\right\Vert _{\infty}\right),
\end{align*}
where $(*)$ used the fact that $\left\langle Y_{\sigma}(\gamma)p,p\right\rangle =\sigma_{\gamma}(p,p)=0$.
Thus
\[
\max_{t\in\mathbb{S}_{\tau}}\left|p(t)\right|\leq\left(\sqrt{\tau}+\frac{\sqrt{2\tau}}{\varepsilon}\right)(\delta+\sqrt{\tau}\left\Vert \nabla V\right\Vert _{\infty}).
\]
\end{proof}
\begin{rem}
The only place in the paper where it is crucial we are working on
a torus $\mathbb{T}^{n}$ is in the preceding lemma. More precisely,
\eqref{eq:lower} and \eqref{eq:upper} use the fact that $\mathbb{T}^{n}$
is parallelizable and has a flat metric $\left\langle \cdot,\cdot\right\rangle $. \end{rem}
\begin{cor}
\label{cor:finite}Suppose $\sigma$ is non-resonant in period $\tau$.
Then for any potential $V\in C^{\infty}(\mathbb{S}_{\tau}\times\mathbb{T}^{2N},\mathbb{R})$,
the set $\mathcal{P}_{\tau}(\sigma,V)$ is compact. Thus if $(\sigma,V,\tau)$
is non-degenerate then $\mathcal{P}_{\tau}(\sigma,V)$ is a finite
set.\end{cor}
\begin{rem}
\label{no non constant solutions}In fact, a similar argument shows
that if $\sigma$ is non-resonant in period $\tau$ then $\mathcal{P}_{\tau}(\sigma,0)$
has no non-constant solutions. Indeed, if $w=(\gamma,p)$ belongs
to $\mathcal{P}_{\tau}(\sigma,0)$ then it follows from \eqref{eq:linftiy bound}
that $p=0$ and that $\gamma$ is constant (cf. \eqref{eq:the symp gradient}).
In particular, this shows that $\mathcal{P}_{\tau}^{h}(\sigma,0)=\emptyset$
for all $h\ne0$. 
\end{rem}
As mentioned in Remark \ref{more general nonres}, Lemma \ref{lem:condition},
Corollary \ref{cor:finite} and Remark \ref{no non constant solutions}
all go through if we take the validity of \ref{eq:alt nonres} to
be the definition of non-resonance.

\subsection{\label{sub:Moduli-spaces}Moduli spaces}

From now on we assume that $\sigma$ is non-resonant in period $\tau$,
and $(\sigma,V,\tau)$ is non-degenerate. Fix two critical points
$[w^{-},z^{-}],[w^{+},z^{+}]\in\mbox{Crit}_{\tau}^{h}(\mathcal{A}_{\sigma,V})$.
Write $\sigma=\sum_{j=1}^{N}\mathtt{p}_{j}^{*}(a_{j}\mu)$, and choose
$A\geq\max_{1\leq j\leq N}\left\Vert a_{j}\right\Vert _{\infty}$.
We will work with paths $\{J_{t}\}_{t\in\mathbb{S}_{\tau}}\subset\mathcal{U}(\sigma,A,\varepsilon_{0})$,
where the set $\mathcal{U}(\sigma,A,\varepsilon_{0})$ was defined
in Definition \ref{unitame def}, and $\varepsilon_{0}>0$ is a small
constant to be specified later. Lemma \ref{lem:uniformly tame} shows
that such paths $\{J_{t}\}$ exist, and moreover that a sufficiently
small perturbation $\{J'_{t}\}$ of $\{J_{t}\}$ still belongs to
$\mathcal{U}(\sigma,A,\varepsilon_{0})$. This means that it makes
sense to talk about a ``generic'' family $\{J_{t}\}\subset\mathcal{U}(\sigma,A,\varepsilon)$.

Denote by 
\[
\mathcal{M}_{\tau}^{h}([w^{-},z^{-}],[w^{+},z^{+}],\sigma,V,\{J_{t}\})
\]
the set of smooth maps $u:\mathbb{R}\times\mathbb{S}_{\tau}\rightarrow T^{*}\mathbb{T}^{2N}$
which satisfy the \emph{Floer equation}\textbf{\emph{:}}
\begin{equation}
\partial_{s}u+J_{t}(u)(\partial_{t}u-X_{\sigma,V}(t,u))=0,\label{eq:floer}
\end{equation}
and which submit to the asymptotic conditions 
\[
\lim_{s\rightarrow\pm\infty}u(s,t)=w^{\pm}(t),\ \ \ \lim_{s\rightarrow\pm\infty}\partial_{s}u(s,t)=0,\ \ \ \mbox{uniformly in }t,
\]
and which in addition satisfy 
\[
[w^{+},z^{-}\sharp(\pi\circ\bar{u})]=[w^{+},z^{+}].
\]
Here $\bar{u}:[0,1]\times\mathbb{S}_{\tau}\rightarrow T^{*}\mathbb{T}^{2N}$
is the compactification of $u$ (such a map $\bar{u}$ exists as the
convergence of $u(s,t)$ to $w^{\pm}(t)$ as $s\rightarrow\pm\infty$
is exponentially fast due to the assumption that $[w^{\pm},z^{\pm}]$
are non-degenerate), and $z^{-}\sharp(\pi\circ\bar{u})$ denotes (a
smooth reparametrization) of the map 
\[
(z^{-}\sharp(\pi\circ\bar{u}))(r,t)=\begin{cases}
z^{-}(2r,t), & 0\leq r\leq1/2,\\
(\pi\circ\bar{u})(2r-1,t), & 1/2\leq r\leq1.
\end{cases}
\]
We can consider $u\in\mathcal{M}_{\tau}^{h}([w^{-},z^{-}],[w^{+},z^{+}],\sigma,V,\{J_{t}\})$
as a map $\widehat{u}:\mathbb{R}\rightarrow\mathbb{L}_{\tau}^{h}(T^{*}\mathbb{T}^{2N},\sigma)$
as follows. Write $u(s,t)=(x(s,t),p(s,t))$, and set 
\begin{equation}
\widehat{u}(s):=[u(s,\cdot),z(s)],\label{eq:uhat}
\end{equation}
where $z(s)$ is (a smooth reparametrization) of the cylinder obtained
by gluing $z^{-}$ onto the cylinder $\{x(r,t)\}_{(r,t)\in(-\infty,s]\times\mathbb{S}_{\tau}}$.
Thus $\widehat{u}$ is a flow line of $\mathcal{A}_{\sigma,V}$: 
\[
\partial_{s}\widehat{u}+\nabla_{J_{t}}\mathcal{A}_{\sigma,V}(\widehat{u}(s))=0,
\]
where $\nabla_{J_{t}}\mathcal{A}_{\sigma,V}$ denotes the vector field
on $\mathbb{L}_{\tau}^{h}(T^{*}\mathbb{T}^{2N},\sigma)$ defined by
\[
\nabla_{J_{t}}\mathcal{A}_{\sigma,V}([w,z]):=J_{t}(w)(\partial_{t}w-X_{\sigma,V}(t,w)).
\]
 
\begin{lem}
\label{lem:claim 1}If $u\in\mathcal{M}_{\tau}^{h}([w^{-},z^{-}],[w^{+},z^{+}],\sigma,V,\{J_{t}\})$
then 
\[
\left\Vert \partial_{s}u\right\Vert _{2}^{2}\leq4\sup_{t\in\mathbb{S}_{\tau}}\tfrac{1}{\kappa_{J_{t}}}\left(\mathcal{A}_{\sigma,V}([w^{-},z^{-}])-\mathcal{A}_{\sigma,V}([w^{+},z^{+}]\right),
\]
where $\kappa_{J_{t}}>0$ denote the minimal positive eigenvalue of
$-\mathtt{J}\circ J_{t}$, as in Lemma \ref{lem:AS bound}. \end{lem}
\begin{proof}
We compute:
\begin{align*}
\left\Vert \partial_{s}u\right\Vert _{2}^{2} & \overset{\textrm{def}}{=}\int_{-\infty}^{\infty}\int_{0}^{\tau}\left\langle \left\langle \partial_{s}u,\partial_{s}u\right\rangle \right\rangle dtds\\
 & \leq\sup_{t\in\mathbb{S}_{\tau}}\frac{1}{\kappa_{J_{t}}}\int_{-\infty}^{\infty}\int_{0}^{\tau}d\lambda(J_{t}\partial_{s}u,\partial_{s}u)udtds\\
 & \leq4\sup_{t\in\mathbb{S}_{\tau}}\frac{1}{\kappa_{J_{t}}}\int_{-\infty}^{\infty}\int_{0}^{\tau}\omega_{\sigma}(J_{t}\partial_{s}u,\partial_{s}u)dtds\\
 & =4\sup_{t\in\mathbb{S}_{\tau}}\frac{1}{\kappa_{J_{t}}}\left(\lim_{s\rightarrow-\infty}\mathcal{A}_{\sigma,V}(\widehat{u}(s))-\lim_{s\rightarrow\infty}\mathcal{A}_{\sigma,V}(\widehat{u}(s))\right)\\
 & =4\sup_{t\in\mathbb{S}_{\tau}}\frac{1}{\kappa_{J_{t}}}\left(\mathcal{A}_{\sigma,V}([w^{-},z^{-}])-\mathcal{A}_{\sigma,V}([w^{+},z^{+}])\right).
\end{align*}

\end{proof}
We now wish to prove that the moduli spaces $\mathcal{M}_{\tau}^{h}([w^{-},z^{-}],[w^{+},z^{+}],\sigma,V,\{J_{t}\})$
have good compactness properties. The next result is an easy consequence
of Lemma \ref{lem:condition}.
\begin{lem}
\label{lem:claim 2}There exists a constant $T>0$ such that if 
\[
u=(x,p)\in\mathcal{M}_{\tau}^{h}([w^{-},z^{-}],[w^{+},z^{+}],\sigma,V,\{J_{t}\})
\]
 then 
\[
\left\Vert p(s,\cdot)\right\Vert _{2}\leq T\left(1+\left\Vert \partial_{s}u(s,\cdot)\right\Vert _{2}\right).
\]

\end{lem}
The following theorem is the central result of this section.
\begin{thm}
\label{thm:AS1}There exists $\varepsilon_{0}>0$ such that if $\{J_{t}\}_{t\in\mathbb{S}_{\tau}}\subset\mathcal{U}(\sigma,A,\varepsilon_{0})$
then for any pair $(C,E)\subseteq\mathbb{R}\times[0,\infty)$, there
exists a compact set $K(A,C,E)\subseteq T^{*}\mathbb{T}^{2N}$ with
the following property: Suppose $[w^{\pm},z^{\pm}]\in\mbox{\emph{Crit}}_{\tau}^{h}(\mathcal{A}_{\sigma,V})$
satisfy
\[
\mathcal{A}_{\sigma,V}([w^{-},z^{-}])\leq C,
\]
\[
\mathcal{A}_{\sigma,V}([w^{-},z^{-}])-\mathcal{A}_{\sigma,V}([w^{+},z^{+}])\leq E.
\]
Then if $u\in\mathcal{M}_{\tau}^{h}([w^{-},z^{-}],[w^{+},z^{+}],\sigma,V,\{J_{t}\})$
one has 
\[
u(\mathbb{R}\times\mathbb{S}_{\tau})\subseteq K(A,C,E).
\]

\end{thm}
Theorem \ref{thm:AS1} can be proved following Abbondandolo and Schwarz'
method in \cite{AbbondandoloSchwarz2006}, as we now explain. The
method has two distinct stages. The first stage appears as Lemma 1.12
in \cite{AbbondandoloSchwarz2006}, and asserts that under the hypotheses
of the theorem, $ $there exists a constant $S(A,C,E)>0$ such that
for any $u=(x,p)$ belonging to $\mathcal{M}_{\tau}^{h}([w^{-},z^{-}],[w^{+},z^{+}],\sigma,V,\{J_{t}\})$,
and any interval $\mathbb{I}\subseteq\mathbb{R}$ it holds that 
\begin{equation}
\left\Vert p|_{\mathbb{I}\times\mathbb{S}_{\tau}}\right\Vert _{2}\leq S(C,E)\left|\mathbb{I}\right|^{1/2},\ \ \ \left\Vert \nabla p|_{\mathbb{I}\times\mathbb{S}_{\tau}}\right\Vert _{2}\leq S(C,E)\left(\left|\mathbb{I}\right|^{1/2}+1\right).\label{eq:first stage}
\end{equation}
A careful inspection of their proof shows that everything apart from
Claim 1 and Claim 2 goes through verbatim in our case. Claim 1 however
is precisely the statement of Lemma \ref{lem:claim 1}, and Claim
2 is precisely the statement of Lemma \ref{lem:claim 2}. The second
stage appears as Theorem 1.14 in \cite{AbbondandoloSchwarz2006}.
The proof then uses Calderon-Zygmund estimates for the Cauchy-Riemann
operator, together with certain interpolation inequalities, to upgrade
equation \eqref{eq:first stage} to the full statement of Theorem
\ref{thm:AS1}. It is this stage that requires $\sup_{t\in\mathbb{S}_{\tau}}\left\Vert J_{t}-\mathtt{J}_{A}\right\Vert $
to be sufficiently small (for some $A>0$), and thus which determines
the constant $\varepsilon_{0}>0$ referred to at the start of this
section. Anyway, provided that this is satisfied, the proof of this
stage goes through word for word in our situation.

\subsection{The Novikov chain complex}

We continue to assume that $\sigma$ is non-resonant in period $\tau$,
and $(\sigma,V,\tau)$ is non-degenerate. For each $w\in\mathcal{P}_{\tau}^{h}(\sigma,V)$,
let $\mu_{\textrm{CZ}}(w)$ denote the \emph{Conley-Zehnder index}\textbf{
}of $w$. In order to define the Conley-Zehnder index we choose a
vertical preserving symplectic trivialization (see \cite{AbbondandoloSchwarz2006});
the fact that $c_{1}(T^{*}M,\omega_{\sigma})=0$ means that that the
value of $\mu_{\textrm{CZ}}(w)$ is independent of this choice of
trivialization. Note however that our sign conventions match those
of \cite{AbbondandoloSchwarz2010} not \cite{AbbondandoloSchwarz2006}.
The non-degeneracy condition\textbf{ }\eqref{eq:non deg eq}\textbf{
}implies that $\mu_{\textrm{CZ}}(w)$ is always an integer. 

Given $j\in\mathbb{Z}$ let 
\[
\mathcal{P}_{\tau}^{h}(\sigma,V)_{j}:=\{w\in\mathcal{P}_{\tau}^{h}(\sigma,V)\mid\mu_{\textrm{CZ}}(w)=j\},
\]
\[
\mbox{Crit}_{\tau}^{h}(\mathcal{A}_{\sigma,V})_{j}:=\{[w,z]\in\mbox{Crit}_{\tau}^{h}(\mathcal{A}_{\sigma,V})\mid\mu_{\textrm{CZ}}(w)=j\}.
\]
It follows from Theorem \ref{thm:AS1} that for a generic family $\{J_{t}\}_{t\in\mathbb{S}_{\tau}}\subset\mathcal{U}(\sigma,A,\varepsilon_{0})$
the moduli spaces $\mathcal{M}_{\tau}^{h}([w^{-},z^{-}],[w^{+},z^{+}],\sigma,V,\{J_{t}\})$
all carry the structure of a $(\mu_{\textrm{CZ}}(w^{-})-\mu_{\textrm{CZ}}(w^{+}))$-dimensional
manifold. Moreover if $\mu_{\textrm{CZ}}(w^{-})=\mu_{\textrm{CZ}}(w^{+})+1$
then the quotient space 
\[
\mathcal{M}_{\tau}^{h}([w^{-},z^{-}],[w^{+},z^{+}],\sigma,V,\{J_{t}\})/\mathbb{R}
\]
 is a finite set. We define the \emph{Novikov Floer chain group} as
\begin{align*}
CF_{j}^{h}(\sigma,V,\tau): & =\left\{ \sum_{k=0}^{\infty}c_{k}[w_{k},z_{k}]\mid\mu_{\textrm{CZ}}(w_{k})=j,\ \mathcal{A}_{\sigma,V}([w_{k},z_{k}])\rightarrow\infty\mbox{ as }k\rightarrow\infty\right\} .\\
 & \cong\mathcal{P}_{\tau}^{h}(\sigma,V)_{j}\otimes\Lambda_{\sigma}^{h}.
\end{align*}
Thus $CF_{j}^{h}(\sigma,V,\tau)$ is a $\#\mathcal{P}_{\tau}^{h}(\sigma,V)_{j}$-dimensional
vector space over $\Lambda_{\sigma}^{h}$ (note our assumptions imply
$\mathcal{P}_{\tau}^{h}(\sigma,V)$ is a finite set, cf. Corollary
\ref{cor:finite}). The boundary operator $\partial_{J_{t}}:CF_{j}\rightarrow CF_{j-1}$
is defined by
\[
\partial_{J_{t}}([w,z]):=\sum_{[w',z']\in\textrm{Crit}_{\tau}^{h}(\mathcal{A}_{\sigma,V})_{j-1}}n([w,z],[w',z'])\,[w',z'],\ \ \ [w,z]\in\mbox{Crit}_{\tau}^{h}(\mathcal{A}_{\sigma,V})_{j},
\]
where
\[
n([w,z],[w',z']):=\#_{2}\mathcal{M}_{\tau}^{h}([w,z],[w',z'],\sigma,V,\{J_{t}\})/\mathbb{R}.
\]
The fact that the boundary operator is well defined (i.e. $\partial_{J_{t}}([w,z])$
is a well defined element of $CF_{j-1}$) is an immediate consequence
of Theorem \ref{thm:AS1}.

A standard Floer-theoretic argument, as explained in \cite[Section 5]{HoferSalamon1995},
tells us that $\partial_{J_{t}}\circ\partial_{J_{t}}=0$, and hence
we may define the \emph{Novikov Floer homology}\textbf{ }$HF_{*}^{h}(\sigma,V,\tau)$
to be the homology of the chain complex $\{CF_{*}^{h}(\sigma,V,\tau),\partial_{J_{t}}\}$.
Moreover $HF_{*}^{h}(\sigma,V,\tau)$ is independent (up to canonical
isomorphism) of the choice family of almost complex structures $\{J_{t}\}_{t\in\mathbb{S}_{\tau}}\subset\mathcal{U}(\sigma,A,\varepsilon_{0})$
(see for instance \cite[Theorem 1.19]{AbbondandoloSchwarz2006}),
which explains why we may safely omit it from our notation.
\begin{rem}
\label{dealing with the degenerate case}Suppose that $\sigma$ is
non-resonant in period $\tau$ but that $(\sigma,V,\tau)$ is degenerate.
By Theorem \ref{thm:transversality} we can make an arbitrarily small
perturbation of the potential $V$ to a new one $V'$ such that $(\sigma,V',\tau)$
is non-degenerate. Moreover if $V''$ is another such perturbation
then by Theorem \ref{thm:continuations} below we have $HF_{*}^{h}(\sigma,V',\tau)\cong HF_{*}^{h}(\sigma,V'',\tau)$.
In other words, we can still define $HF_{*}^{h}(\sigma,V,\tau)$ even
when $(\sigma,V,\tau)$ is degenerate, by simply setting 
\[
HF_{*}^{h}(\sigma,V,\tau)\overset{\textrm{def}}{=}HF_{*}^{h}(\sigma,V',\tau)
\]
 for any potential $V'$ such that $\left\Vert V-V'\right\Vert _{\infty}$
is sufficiently small and such that $(\sigma,V',\tau)$ is non-degenerate. 
\end{rem}

\subsection{Invariance}

In this section we prove invariance through exact deformations of
the magnetic form $\sigma$, and deformations of the potential $V$.
Fix a\textbf{ }1-form $\theta\in\Omega^{1}(\mathbb{T}^{2N})$, and
set 
\[
\sigma_{s}:=\sigma+sd\theta.
\]
Fix $\tau>0$ and assume that:
\begin{itemize}
\item $\sigma_{s}$ is non-resonant in period $\tau$ for all $s\in[0,1]$.
\end{itemize}
Note that $\Lambda_{\sigma_{s}}^{h}\equiv\Lambda_{\sigma}^{h}$ for
all $s\in[0,1]$ (cf. Remark \ref{projective line}). Suppose $V_{0},V_{1}\in C^{\infty}(\mathbb{S}_{\tau}\times\mathbb{T}^{2N},\mathbb{R})$.
Set 
\[
V_{s}:=(1-s)V_{0}+sV_{1}.
\]
Assume that:
\begin{itemize}
\item For generic $s\in[0,1]$, and in particular for $s=0,1$, all the
triples $(\sigma_{s},V_{s},\tau)$ are non-degenerate.
\end{itemize}
Under these conditions we have proved that the Floer homologies $HF_{*}^{h}(\sigma_{0},V_{0},\tau)$
and $HF_{*}^{h}(\sigma_{1},V_{1},\tau)$ are both well defined. We
now wish to prove they are isomorphic.
\begin{thm}
\label{thm:continuations}Under the above assumptions there exists
a continuation map 
\[
\Psi:CF_{*}^{h}(\sigma_{0},V_{0},\tau)\rightarrow CF_{*}^{h}(\sigma_{1},V_{1},\tau)
\]
inducing an isomorphism 
\begin{equation}
\psi:HF_{*}^{h}(\sigma_{0},V_{0},\tau)\rightarrow HF_{*}^{h}(\sigma_{1},V_{1},\tau).\label{eq:first iso}
\end{equation}

\end{thm}
The key ingredient needed to prove Theorem \ref{thm:continuations}
is the following proposition, which proves \emph{energy estimates}\textbf{
}for certain $s$-dependent trajectories. Fix a smooth cutoff function
$\beta:\mathbb{R}\rightarrow[0,1]$ satisfying $\beta(s)\equiv0$
for $s\leq0$ and $\beta(s)\equiv1$ for $s\geq1$, with $0\leq\beta'(s)\leq2$
for all $s\in\mathbb{R}$.
\begin{prop}
\label{prop:energy estimates}There exists $\delta>0$ with the following
property. Suppose $0=s_{0}<s_{1}<\dots<s_{N}=1$ satisfies
\[
\max_{j}(s_{j+1}-s_{j})<\delta,
\]
and set
\[
\beta_{j}(s):=s_{j}+\beta(s)(s_{j+1}-s_{j}),
\]
\[
\sigma_{s}^{j}:=\sigma+\beta_{j}(s)d\theta,\ \ \ \sigma^{j}:=\sigma_{s_{j}},
\]
\[
V_{s}^{j}(t,x):=V_{\beta_{j}(s)}(t,x),\ V^{j}:=V_{s_{j}}.
\]
Then given any $E\in\mathbb{R}$ there exists $R(E)>0$ with the following
property: for any $j=0,1,\dots,N-1$, if $[w^{-},z^{-}]\in\mbox{\emph{Crit}}_{\tau}^{h}(\mathcal{A}_{\sigma^{j},V^{j}})$
and $[w^{+},z^{+}]\in\mbox{\emph{Crit}}_{\tau}^{h}(\mathcal{A}_{\sigma^{j+1},V^{j+1}})$
satisfy
\[
\mathcal{A}_{\sigma^{j},V^{j}}([w^{-},z^{+}])-\mathcal{A}_{\sigma^{j+1},V^{j+1}}([w^{+},z^{+}])\leq E,
\]
then given any solution 
\[
u\in\mathcal{M}_{\tau}^{h}([w^{-},z^{-}],[w^{+},z^{+}],\sigma_{s}^{j},V_{s}^{j},\{J_{t}\})
\]
(where the moduli space of $s$-dependent solutions is defined analogously
to before) it holds that
\[
\left\Vert \partial_{s}u\right\Vert _{2}^{2}\leq R(E).
\]
\end{prop}
\begin{proof}
Fix $j\in\{0,1,\dots,N-1\}$. As in Section \ref{sub:Moduli-spaces},
it is convenient to interpret a map $u(s,t)=(x(s,t),p(s,t))$ as in
the statement of the proposition also as a map $\widehat{u}:\mathbb{R}\rightarrow\mathbb{L}_{\tau}^{h}(T^{*}\mathbb{T}^{2N},\sigma)$
as in \eqref{eq:uhat}. Thus $\widehat{u}$ is a flow line of $\mathcal{A}_{\sigma_{s}^{j},V_{s}^{j}}$:
\[
\partial_{s}\widehat{u}+\nabla_{J_{t}}\mathcal{A}_{\sigma_{s}^{j},V_{s}^{j}}(\widehat{u}(s))=0.
\]
Set
\[
\Delta(u):=\int_{-\infty}^{\infty}\left|\left(\frac{\partial}{\partial s}\mathcal{A}_{\sigma_{s}^{j},V_{s}^{j}}\right)(\widehat{u}(s))\right|ds.
\]
We first bound $\left|\Delta(u)\right|$ in terms of $\left\Vert \partial_{s}u\right\Vert _{2}^{2}$.
Note that
\[
\left|\left(\frac{\partial}{\partial s}\mathcal{A}_{\sigma_{s}^{j},V_{s}^{j}}\right)(\widehat{u}(s))\right|\leq\left|\int_{0}^{\tau}\left(\frac{\partial}{\partial s}V_{s}^{j}\right)(u(s,\cdot))dt\right|+\left|\left(\frac{\partial}{\partial s}\mathcal{A}_{\sigma_{s}^{j}}\right)([x(s,\cdot),z(s)])\right|.
\]
We can estimate the first term via
\[
\left|\int_{0}^{\tau}\left(\frac{\partial}{\partial s}V_{s}^{j}\right)(u(s,\cdot))dt\right|\leq2\delta\tau\left\Vert V_{1}-V_{0}\right\Vert _{\infty}.
\]
For the second term we have 
\begin{align*}
\left|\left(\frac{\partial}{\partial s}\mathcal{A}_{\sigma_{s}^{j}}\right)([x(s,\cdot),z(s)])\right| & \leq2\delta\left|\int_{\mathbb{S}_{\tau}}x(s,\cdot)^{*}\theta\right|+2\delta\left|\int_{S^{1}}\gamma_{h}^{*}\theta\right|\\
 & \leq2\delta\left\Vert \theta\right\Vert _{\infty}\left\Vert \partial_{t}x(s,\cdot)\right\Vert _{1}+2\delta\left|\int_{S^{1}}\gamma_{h}^{*}\theta\right|,
\end{align*}
where the reference loops $\gamma_{h}$ were defined in \eqref{eq:reference loop}
and 
\[
\left\Vert \partial_{t}x(s,\cdot)\right\Vert _{1}:=\int_{0}^{\tau}\left|\partial_{t}x(s,\cdot)\right|dt.
\]
We now estimate
\begin{align*}
\left\Vert \partial_{t}x(s,\cdot)\right\Vert _{1} & \leq\sqrt{\tau}\left\Vert \partial_{t}x(s,\cdot)\right\Vert _{2}\\
 & \leq\sqrt{\tau}\left(1+\left\Vert \partial_{t}x(s,\cdot)\right\Vert _{2}^{2}\right)\\
 & \leq\sqrt{\tau}\left(1+\sup_{t\in\mathbb{S}_{\tau}}\left\Vert J_{t}\right\Vert _{\infty}^{2}\left\Vert \partial_{s}u(s,\cdot)\right\Vert _{2}^{2}+\left\Vert p(s,\cdot)\right\Vert _{2}^{2}\right),
\end{align*}
where the last line follows from taking horizontal components of the
equation 
\[
\partial_{t}u=J_{t}(u)\partial_{s}u+X_{\sigma_{s}^{j},V_{s}^{j}}(t,u)
\]
(cf. \eqref{eq:the symp gradient}). Set 

\[
C_{1}:=2\sqrt{\tau}\left\Vert \theta\right\Vert _{\infty}
\]
\[
C_{2}:=\sup_{t\in\mathbb{S}_{\tau}}\left\Vert J_{t}\right\Vert _{\infty}^{2},
\]
\[
C_{3}:=2\sqrt{\tau}\left\Vert \theta\right\Vert _{\infty}+2\tau\left\Vert V_{1}-V_{0}\right\Vert _{\infty}+2\left|\int_{S^{1}}\gamma_{h}^{*}\theta\right|.
\]
We have shown 
\[
\Delta(u)\leq C_{1}\delta\left\Vert p|_{[0,1]\times\mathbb{S}_{\tau}}\right\Vert _{2}^{2}+C_{1}C_{2}\delta\left\Vert \partial_{s}u\right\Vert _{2}^{2}+C_{3}\delta.
\]
It follows from Lemma \ref{lem:condition} that there exists a constant
$T>0$ such that 
\[
\left\Vert p|_{[0,1]\times\mathbb{S}_{\tau}}\right\Vert \leq T(1+\left\Vert \partial_{s}u|_{[0,1]\times\mathbb{S}_{\tau}}\right\Vert )
\]
(cf. Lemma \ref{lem:claim 2}), and hence
\[
\left\Vert p|_{[0,1]\times\mathbb{S}_{\tau}}\right\Vert ^{2}\leq3T^{2}(1+\left\Vert \partial_{s}u|_{[0,1]\times\mathbb{S}_{\tau}}\right\Vert ^{2}).
\]
and hence we can estimate
\[
\Delta(u)\leq(3C_{1}T^{2}+C_{1}C_{2})\delta\left\Vert \partial_{s}u\right\Vert _{2}^{2}+(3C_{1}T^{2}+C_{3})\delta.
\]
Next, the conclusion of Lemma \ref{lem:claim 1} becomes: 
\[
\left\Vert \partial_{s}u\right\Vert _{2}^{2}\leq4\sup_{t\in\mathbb{S}_{\tau}}\tfrac{1}{\kappa_{J_{t}}}(E+\Delta(u)),
\]
where as $\kappa_{J_{t}}$ was defined in Lemma \ref{lem:AS bound}.
Thus provided we choose $\delta$ small enough such that
\[
4\sup_{t\in\mathbb{S}_{\tau}}\tfrac{1}{\kappa_{J_{t}}}(3C_{1}T^{2}+C_{1}C_{2})\delta<\tfrac{1}{2},
\]
we obtain 
\[
\left\Vert \partial_{s}u\right\Vert _{2}^{2}\leq8\sup_{t\in\mathbb{S}_{\tau}}\tfrac{1}{\kappa_{J_{t}}}(E+(3C_{1}T^{2}+C_{3})\delta).
\]
This completes the proof.
\end{proof}
Proposition \ref{prop:energy estimates} is precisely what is needed
in order to extend Theorem \ref{thm:AS1} to $s$-dependent trajectories
(see \cite[Lemma 1.21]{AbbondandoloSchwarz2006} - in particular the
statements of Claim 1' and Claim 2'). One now applies a standard adiabatic
argument to complete the proof of Theorem \ref{thm:continuations};
see for instance \cite{Salamon1999} for more details.

\section{\label{sec:Computing-the-Novikov}Computing the Novikov Floer homology}

Having defined the Floer homology $HF_{*}^{h}(\sigma,V,\tau)$ for
any magnetic field $\sigma$ which is non-resonant in period $\tau$,
and any potential $V$ (if $(\sigma,V,\tau)$ is degenerate then one
first perturbs $V$, as in Remark \ref{dealing with the degenerate case}),
we now proceed to compute it. 

Suppose we work on $\mathbb{T}^{2}$ with $\sigma=a_{0}\mu$ for some
$a_{0}\in\mathbb{R}$ such that $a_{0}\tau\notin2\pi\mathbb{Z}$.
By Remark \ref{no non constant solutions} one has $\mathcal{P}_{\tau}^{h}(\sigma,0)=\emptyset$
whenever $h\ne0$, and for $h=0$ one has $\mathcal{P}_{\tau}^{h}(\sigma,0)\cong\mathbb{T}^{2}$.
This immediately implies that $HF_{*}^{h}(\sigma,0,\tau)=0$ for $h\ne0$,
since in this case $(\sigma,0,\tau)$ is vacuously non-degenerate
when restricted to non-contractible critical points. In the contractible
component $(\sigma,0,\tau)$ is \emph{not} non-degenerate on the contractible
component, since the critical set is diffeomorphic to $\mathbb{T}^{2}$.
So far if we wanted to compute $HF_{*}^{0}(\sigma,0,\tau)$ we would
first choose a small perturbation $V^{\varepsilon}$ and then define
\begin{equation}
HF_{*}^{0}(\sigma,0,\tau)\overset{\textrm{def}}{=}HF_{*}^{0}(\sigma,V^{\varepsilon},\tau).\label{eq:def}
\end{equation}
However there is a much easier method. Indeed, whilst $\mathcal{A}_{\sigma,0}$
is not a Morse function on $\mathcal{L}_{\tau}^{0}T^{*}\mathbb{T}^{2}$,
it \emph{is }a \emph{Morse-Bott }function. This gives an alternative
way to compute $HF_{*}^{0}(\sigma,0,\tau)$. One first picks an additional
Morse function $f$ of the critical point set (in this case, a Morse
function $f$ on $\mathbb{T}^{2}$), and then counts \emph{gradient
flow lines with cascades }of the pair $(\mathcal{A}_{\sigma,0},f)$.
We emphasize that this is a particularly simple instance of Morse-Bott
Floer homology, since there the critical manifold is connected. The
correct grading to assign in the Morse-Bott case is given by 
\[
\mu_{f}(w):=\mu_{\textrm{CZ}}(w)-\tfrac{1}{2}\dim_{w}\mbox{Crit}(\mathcal{A}_{a\mu,0})+\mbox{ind}_{f}(w),\ \ \ \mbox{for }w\in\mbox{Crit}(f)\subset\mbox{Crit}(\mathcal{A}_{a\mu,0}),
\]
where $\dim_{w}\mbox{Crit}(\mathcal{A}_{a\mu,0})$ denotes the local
dimension of $\mbox{Crit}(\mathcal{A}_{a\mu,0})$ at $w$, and $\mbox{ind}_{f}(w)$
denotes the Morse index of $w$ as a critical point of $f$. In our
case this simplifies to 
\[
\mu_{f}(w)=\mu_{0}-1+\mbox{ind}_{f}(w),
\]
where $\mu_{0}\in\mathbb{Z}$ is the common value 
\begin{equation}
\mu_{0}=\mu_{\textrm{CZ}}(\mbox{constant loop}).\label{eq:mu 0}
\end{equation}

Anyway, this gives a new Floer homology group $HF_{*}^{0}(\sigma,0,\tau)^{\textrm{MB}}$,
which does not depend on the choice of Morse function. Moreover for
any sufficiently small perturbation $V^{\varepsilon}$, one has 
\[
HF_{*}^{0}(\sigma,0,\tau)^{\textrm{MB}}\cong HF_{*}^{0}(\sigma,V^{\varepsilon},\tau).
\]
We refer the reader to \cite[Appendix A]{Frauenfelder2004} or \cite[Section 2.3]{BaeFrauenfelder2010}
for more information. Since the critical manifold is simply $\mathbb{T}^{2}$
in our case, we immediately obtain 
\begin{equation}
HF_{*}^{0}(\sigma,0,\tau)\cong H_{*+\mu_{0}-1}(\mathbb{T}^{2};\mathbb{Z}_{2})\label{eq:index d}
\end{equation}
where $\mu_{0}$ was defined in \eqref{eq:mu 0}. In Lemma \ref{lem:fourier}
below we show that $\mu_{0}=2k+1$ where $k$ is the unique integer
such that $2\pi k<\left|a_{0}\right|\tau<2\pi(k+1)$. 

Now suppose $\sigma$ is non-resonant in period $\tau$ and $V$ is
any potential such that $(\sigma,V,\tau)$ is non-degenerate. Let
$a_{0}:=\int_{\mathbb{T}^{2}}\sigma$. Then $a_{0}\mu-\sigma$ is
exact, say $a_{0}\mu-\sigma=d\theta$. Thus if we set $\sigma_{s}:=\sigma+sd\theta$,
and we choose a generic homotopy from $V_{s}$ from $V=V_{0}$ to
$V^{\varepsilon}=V_{1}$, we can apply Theorem \ref{thm:continuations},
together with \eqref{eq:def} and \eqref{eq:index d}, to deduce that
\[
HF_{*}^{h}(\sigma,V,\tau)\cong HF_{*}^{h}(a_{0}\mu,0,\tau)\cong\begin{cases}
H_{*+2k}(\mathbb{T}^{2};\mathbb{Z}), & h=0,\\
0, & h\ne0,
\end{cases},
\]
where as before $k$ is the unique integer such that $2\pi k<\left|a_{0}\right|\tau<2\pi(k+1)$.
Theorems A and B now follow by standard arguments. Indeed, if $(\sigma,V,\tau)$
is non-degenerate then since $\mbox{rank}\, HF_{*}^{0}(\sigma,V,\tau)=\mbox{rank}\, H_{*+2k}(\mathbb{T}^{2};\mathbb{Z}_{2})=4$,
we immediately see that $\#\mathcal{P}_{\tau}^{0}(\sigma,V)$ is at
least four. The argument is more involved in the degenerate case,
but standard; see for instance \cite{LeOno1996} or \cite{AlbersHein2013}. 

Finally to deal with the case $N>1$ we first argue as above to see
that if $\sigma=\sum_{j=1}^{N}\mathtt{p}_{j}^{*}(a_{j}\mu)$ for constants
$a_{j}\in\mathbb{R}$ such that $a_{j}\tau\notin2\pi\mathbb{Z}$ then
\[
HF_{*}^{h}(\sigma,0,\tau)=\begin{cases}
H_{*+2k}(\mathbb{T}^{2N};\mathbb{Z}), & h=0,\\
0, & h\ne0,
\end{cases}
\]
where $k=\sum_{j=1}^{N}k_{j}$ and $k_{j}$ is the unique integer
such that $2\pi k_{j}<\left|a_{j}\right|\tau<2\pi(k_{j}+1)$. Then
in the general case where $\sigma=\sum_{j=1}^{N}\mathtt{p}_{j}^{*}\sigma_{j}$,
with each $\sigma_{j}\in\Omega^{2}(\mathbb{T}^{2})$ non-resonant,
we use the exact non-resonant deformation $\sigma_{s}:=\sigma+sd\theta$,
where $\theta=\sum_{j=1}^{N}\mathtt{p}_{j}^{*}\theta$, and $\theta_{j}$
1-form on $\mathbb{T}^{2N}$ satisfying
\[
d\theta_{j}=\left(\int_{\mathbb{T}^{2}}\sigma_{j}\right)\mu-\sigma_{j}.
\]
Now Theorems A+ and B+ follow similarly.

\section{\label{sec:The-Lagrangian-Setting}The Lagrangian setting}

In this section we briefly outline the ``Lagrangian'' method, with
the aim of explaining Remark \ref{lagrangian remark} from the Introduction.
In this setting rather than working with the free loop space $\mathcal{L}_{\tau}\mathbb{T}^{2N}$
we work with its \emph{Sobolev completion }$\widetilde{\mathcal{L}}_{\tau}\mathbb{T}^{2N}:=W^{1,2}(\mathbb{S}_{\tau},\mathbb{T}^{2N})$.
Unlike $\mathcal{L}_{\tau}\mathbb{T}^{2N}$, the space $\widetilde{\mathcal{L}}_{\tau}\mathbb{T}^{2N}$
carries the structure of a Hilbert manifold, and therefore is much
better suited for doing Morse homology. In this section we use the
Hilbert product $\left\langle \cdot,\cdot\right\rangle _{1,2}$ on
$\widetilde{\mathcal{L}}_{\tau}\mathbb{T}^{2N}$ defined by 
\[
\left\langle \xi,\zeta\right\rangle _{1,2}:=\int_{0}^{\tau}\left\langle \xi(t),\zeta(t)\right\rangle dt+\int_{0}^{\tau}\left\langle \nabla_{t}\xi,\nabla_{t}\zeta\right\rangle dt.
\]
We restrict our attention to the contractible component $\widetilde{\mathcal{L}}_{\tau}^{0}\mathbb{T}^{2N}$;
thus the functional $\mathcal{A}_{\sigma}$ from \eqref{eq:a sigma}
is well defined on $\widetilde{\mathcal{L}}_{\tau}^{0}\mathbb{T}^{2N}$
itself (cf. Remark \ref{weakly exact remark}). As before fix a time
dependent potential $V\in C^{\infty}(\mathbb{S}_{\tau}\times\mathbb{T}^{2N},\mathbb{R})$
and set 
\[
L_{V}(t,x,v):=\frac{1}{2}\left|v\right|^{2}-V(t,x).
\]
The \emph{Lagrangian action functional }$\mathcal{S}_{\sigma,V}$
is defined as the sum
\begin{align*}
\mathcal{S}_{\sigma,V}(\gamma): & =\mathcal{S}_{V}(\gamma)+\mathcal{A}_{\sigma}(\gamma),
\end{align*}
where $\mathit{\mathcal{S}_{V}}$ is the \emph{standard}\textbf{ }Lagrangian
action functional

\[
\mathcal{S}_{V}(\gamma):=\int_{0}^{\tau}L_{V}(t,\gamma(t),\partial_{t}\gamma(t))dt.
\]
The following lemma is straightforward.
\begin{lem}
A loop $\gamma\in\widetilde{\mathcal{L}}_{\tau}^{0}\mathbb{T}^{2N}$
is a critical point of $\mathcal{S}_{\sigma,V}$ if and only if there
exists $w\in\mathcal{P}_{\tau}^{0}(\sigma,V)$ such that $\pi\circ w=\gamma$.
\end{lem}
Recall that a $C^{1}$-functional $\mathcal{S}:\mathcal{M}\rightarrow\mathbb{R}$
on a Riemannian Hilbert manifold $\mathcal{M}$ satisfies the \emph{Palais-Smale
condition}\textbf{ }if every sequence $(x_{k})\subseteq\mathcal{M}$
for which $\mathcal{S}(x_{k})$ is bounded and $\left\Vert d\mathcal{S}(x_{k})\right\Vert \rightarrow0$
admits a convergent subsequence (here $\left\Vert \cdot\right\Vert $
denotes the dual norm on $T_{x_{k}}^{*}\mathcal{M}$). We wish to
prove:
\begin{thm}
\label{thm:PS}Suppose $\sigma$ is non-resonant in period $\tau$.
Then for any potential $V\in C^{\infty}(\mathbb{S}_{\tau}\times\mathbb{T}^{2N},\mathbb{R})$,
the functional $\mathcal{S}_{\sigma,V}$ satisfies the Palais-Smale
condition. 
\end{thm}
To see this note that 
\[
d\mathcal{S}_{\sigma,V}(\gamma)[\xi]=\int_{0}^{\tau}\left\langle -\nabla_{t}\partial_{t}\gamma-Y_{\sigma}(\gamma)\partial_{t}\gamma-\nabla V_{t}(\gamma),\xi\right\rangle dt,
\]
and hence if $(\gamma_{k})$ is a sequence such that $\left\Vert d\mathcal{S}_{\sigma,V}(\gamma_{k})\right\Vert =o(1)$
then 
\[
\int_{0}^{\tau}\left|\nabla_{t}\partial_{t}\gamma_{k}+Y_{\sigma}(\gamma_{k})\partial_{t}\gamma_{k}+\nabla V_{t}(\gamma_{k})\right|^{2}dt=o(1).
\]
Thus by an argument very similar to Lemma \ref{lem:condition} one
has that 
\[
\int_{0}^{\tau}\left|\partial_{t}\gamma_{k}\right|^{2}dt=O(1).
\]
Now the standard argument, which was originally due to Benci \cite{Benci1986},
goes through, as explained in \cite[Appendix A]{FrauenfelderMerryPaternain2012}.\newline

Ideally one would like to use $\mathcal{S}_{\sigma,V}$ to define
a \emph{Morse complex }$CM_{*}^{0}(\sigma,V,\tau)$ whose generators
are the critical points of $\mathcal{S}_{\sigma,V}$. The homology
of this complex should compute the singular homology of the space
$\widetilde{\mathcal{L}}_{\tau}^{0}\mathbb{T}^{2N}$. Moreover one
expects that when defined, the homology $HM_{*}^{0}(\sigma,V,\tau)$
should be isomorphic to the Floer homology $HF_{*}^{0}(\sigma,V,\tau)$.
In our setting such a construction is possible if and only if the
functional $\mathcal{S}_{\sigma,V}$ is bounded below. We refer the
reader to \cite{AbbondandoloMajer2006} for more information on the
Morse complex, and to \cite{AbbondandoloSchwarz2006} for the idea
behind the isomorphism between the Morse and Floer homologies. Here
we note only the following point.
\begin{lem}
Take $N=1$ and $\sigma=a\mu$ for some $a\in\mathbb{R}$. Assume
that $a\tau\notin2\pi\mathbb{Z}$, and consider the functional $\mathcal{S}_{a\mu,V}:\widetilde{\mathcal{L}}_{\tau}^{0}\mathbb{T}^{2N}\rightarrow\mathbb{R}$.
Then the functional $\mathcal{S}_{a\mu,V}$ is bounded below if and
only if $\left|a\tau\right|<2\pi$. \end{lem}
\begin{proof}
The fact that $\mathcal{S}_{a\mu,V}$ is bounded below for $\left|a\tau\right|<2\pi$
was proved in \cite[Appendix A]{FrauenfelderMerryPaternain2012}.
Let us show that $\mathcal{S}_{a\mu,V}$ is not bounded below if $\left|a\tau\right|>2\pi$.
It suffices to consider the case $V=0$. Let $\widetilde{\gamma}_{R}:\mathbb{S}_{\tau}\rightarrow\mathbb{R}^{2N}$
denote a circle of radius $R$ with centre the origin (explicitly
$\widetilde{\gamma}_{R}(t):=Re^{2\pi it/\tau}$), and let $\gamma_{R}:=\mathtt{q}\circ\widetilde{\gamma}_{R}$,
where $\mathtt{q}:\mathbb{R}^{2N}\rightarrow\mathbb{T}^{2N}$ denotes
the projection map of the universal cover. Then 
\[
\mathcal{S}_{a\mu,0}(\gamma_{R})=\frac{\pi R^{2}}{\tau}\left(\frac{2\pi}{\tau}-a\right),
\]
which for $a\tau>2\pi$ tends to $-\infty$ as $R\rightarrow\infty$
(if $a<0$ one should consider $t\mapsto\widetilde{\gamma}_{R}(\tau-t)$
instead).
\end{proof}
Finally, we compute the index jump $d$ referred to in \eqref{eq:index d}.
\begin{lem}
\label{lem:fourier}Take $N=1$ and $\sigma=a\mu$ for some $a\in\mathbb{R}$.
Assume that $a\tau\notin2\pi\mathbb{Z}$, and let $k_{0}\in\mathbb{Z}$
denote the unique integer such that 
\[
2\pi k_{0}<\left|a\right|\tau<2\pi(k_{0}+1).
\]
Consider the functional $\mathcal{A}_{a\mu,0}:\mathcal{L}_{\tau}^{0}\mathbb{T}^{2}\rightarrow\mathbb{R}$.
Then the Conley-Zehnder index $\mu_{0}$ of a constant solution is
given by 
\[
\mu_{0}=2k_{0}+1.
\]
\end{lem}
\begin{proof}
For simplicity we take $\tau=1$ and assume $a>0$ in what follows.
It is easier to make the computation on the Lagrangian side, and then
use the classical result due originally to \cite{Duistermaat1976}
that relates the Conley-Zehnder index with the Morse index of the
corresponding solution on the Lagrangian side. 

Moreover to compute the Morse index it is convenient to use a Fourier
expansion. If $\gamma\in\mathcal{L}_{0}\mathbb{T}^{2}$ then we can
write 
\[
\gamma(t)=\sum_{k\in\mathbb{Z}}e^{2\pi kt\mathtt{j}}\gamma_{k}.
\]
Then one easily checks that 
\[
\mathcal{S}_{a\mu,0}(\gamma)=\sum_{k\in\mathbb{Z}}\left(2\pi^{2}k^{2}-a\pi k\right)\left|\gamma_{k}\right|^{2}.
\]
From this we can read off that the Morse index of a constant solution
is
\[
\mbox{ind}_{\mathcal{S}_{a\mu,0}}(\mbox{constant})=2\cdot\#\left\{ k\in\mathbb{Z}\mid(2\pi k-a)k<0\right\} =2k_{0},
\]
where the $2$ comes from the fact that the Fourier series have complex
coefficients. The relation between this Morse index and the Conley-Zehnder
index is given by 
\[
\mu_{\textrm{CZ}}(w)=\mbox{ind}_{\mathcal{S}_{a\mu,0}}(\pi\circ w)+\frac{1}{2}\dim_{w}\mbox{Crit}(\mathcal{A}_{a\mu,0})
\]
(see \cite[Section 4]{AbbondandoloPortaluriSchwarz2008} for a detailed
proof of this result in the case $\sigma=0$, and see \cite[Theorem 4.4]{Merry2011a}
for justification as to why this result still holds for non-zero $\sigma$).
Thus we have
\[
\mu_{0}=2k_{0}+1
\]
as claimed. 
\end{proof}
\bibliographystyle{amsalpha}
\bibliography{\string"/Users/iLyX/Dropbox/Other useful files/willmacbibtex\string"}

\end{document}